\documentclass{amsart}
\usepackage{amssymb}
\newtheorem{thm}{Theorem}[section]

\newtheorem{lem}[thm]{Lemma}
\newtheorem{prop}[thm]{Proposition}

\theoremstyle{remark}
\newtheorem{rem}{Remark}[section]

\newcommand{\be}{\begin{equation}}
\newcommand{\ee}{\end{equation}}

\newcommand{\bee}{\begin{equation*}}
\newcommand{\eee}{\end{equation*}}

\newcommand{\bea}{\begin{eqnarray}}
\newcommand{\eea}{\end{eqnarray}}

\newcommand{\Bea}{\begin{eqnarray*}}
\newcommand{\Eea}{\end{eqnarray*}}

\def\CP{{\mathcal P}}
\def\CH{{\mathcal H}}

\def\C{{\mathbb C}}

\def\R{{\mathbb R}}
\def\Z{{\mathbb Z}}

\def\ol{\overline}
\def\ul{\underline}

\newcommand{\wt}{\widetilde}
\newcommand{\wh}{\widehat}

\begin{document}

\title[$L^p$-Integrability and Dimensions]
{ $L^p$-Integrability, Dimensions of Supports of Fourier
transforms and applications}
\author{K. S. Senthil Raani}

\address{Department of Mathematics\\
Indian Institute of Science\\
Bangalore-560 012}

\email{raani@math.iisc.ernet.in}

%\date{\today}

%%%%%%%%%%%ABSTRACT%%%%%%%%%%%%%%%%

\begin{abstract}
It is proved that there does not exist any non zero function in
$L^p(\R^n)$ with $1\leq p\leq 2n/\alpha$ if its Fourier transform
is supported by a set of finite packing $\alpha$-measure where
$0<\alpha<n$. It is shown that the assertion fails for
$p>2n/\alpha$. The result is applied to prove $L^p$ Wiener
Tauberian theorems for $\R^n$ and $M(2)$.
\end{abstract}
\maketitle

%%%%%%%%%%%SECTION 1 INTRODUCTION%%%%%%%%%%%%%%%%%%%%

\section{Introduction}
\setcounter{equation}{0}

A classical result of Wiener \cite{NW} states that the translates
of a function $f \in L^1(\R^n)$ span a dense subset of $L^1(\R^n)$
if and only if the Fourier transform, $\wh{f}$ of $f$ is not zero
at any point on $\R^n$. That is, if $f \in L^1(\R^n)$ and \bee
\wh{f}(t)=\frac{1}{(\sqrt{2\pi})^n}\int_{\R^n}f(x)e^{-ix.t}dx,
\eee then for $^xf(y)=f(-x+y)$, we have $\ol{span\ \{^xf:
x\in\R^n\}}=L^1(\R^n)$ if and only if $\wh{f}(t)\neq0\ \forall\
t\in\R^n$. In fact, if $g\in L^{\infty}(\R^n)$ is such that
$\int_{\R^n}\ ^xf(y)g(y)dy=0\ \forall\ x\in\R^n$, we get
$\wt{f}*g=0$ where $\wt{f}(t)=f(-t)$. Distribution theory tells us
that $supp\ \wh{\wt{g}}\subseteq\{x\in\R^n:\wh{f}(x)=0\}$ (which
is Wiener Tauberian theorem in disguise. See \cite{R}). If
$\wh{f}$ is nowhere vanishing then it follows that $g\equiv0$.
This crucial step in the proof of Wiener's theorem leads us to the
study of functions
$f$ in $L^p(\R^n)$ with $supp\ \wh{f}$ in a thin set.\\

    This question also arises in PDE. If $u$ is a tempered solution of the
equation $P(D)u=0$, where $P(D)$ is a constant coefficient
differential operator then the Fourier transform $\hat{u}$ is
supported in the zero set of the polynomial $P$. Hormander and
Agmon studied the asymptotic properties of $u$ \cite{HA}. M. L.
Agranovsky and E. K. Narayanan in \cite{NA} proved that if $f\in
L^p(\R^n)$ and $supp\ \wh{f}$ is carried by a $C^1$-manifold of
dimension $d<n$, then $f\equiv0$ if $p\leq\frac{2n}{d}$. Notice
that in these results and other ones of similar nature, only the
range $p > 2$ is interesting. If $f \in L^p(\mathbb R^n), 1 \leq p
\leq 2$ and $ \widehat{f}$ is supported in a set of measure zero
then $f$ is identically zero. If $p > 2,$ then $\widehat{f}$ is a
tempered distribution and the support of
$f$ is a closed set which may be thin.\\

Our aim in the first part of this paper is to extend the above
results where the integer dimension manifold set is replaced with
finite packing measurable set (see (\ref{EqnIntro1})). We also
mention that an older result of Beurling (see \cite{B}) says that
if $f \in L^p(\mathbb R), p
>2 $ and $\hat{f}$ is supported by a set of Hausdorff dimension less than $2/p$, then the function is identically zero. We show that \emph{if
$f\in L^p(\R^n)$ and $supp\ \wh{f}$ is contained in a set $E$,
which has a finite packing $\alpha$-measure (for $0<\alpha<n$),
then $f\equiv0$ if $p\leq\frac{2n}{\alpha}$.} By considering the
cartesian product of the Salem set in $\R$(see \cite{Sa} and also
page 263 in \cite{W}), we show that our result is sharp.\\

    In the second part of the paper we use the above result to prove some $L^p$-Wiener Tauberian theorems.
N. Wiener \cite{NW} characterized the cyclic vectors (with respect
to translations) in $L^p(\R)$, for $p = 1, 2$, in terms of the
zero set of the Fourier transform. He conjectured that a similar
characterization should be true for $1 < p < 2$(See page 93 in
\cite{NW}). Nir Lev and Alexander Olevskii in \cite{L} recently
proved that for any $1 < p < 2$ one can find two functions in
$L^1(\R)\cap C_0(\R)$, such that one is cyclic in $L^p(\R)$ and
the other is not, but their Fourier transforms have the same
(compact) set of zeros. This disproves Wiener's conjecture. As is
well known, there are no complete answers to
$L^p$-Weiner-Tauberian theorems when $p\neq1,2$. See pages 234-236
in \cite{W} for initial results and \cite{L} for
more references.\\

    In \cite{B}, A. Beurling proved that if the
Hausdorff dimension of the closed set where the Fourier transform
of $f$ vanishes, is $\alpha$, for $0\leq\alpha\leq1$, then the
space of finite linear combinations of translates of $f$ is dense
in $L^p(\R)$ for $2/(2-\alpha)<p$. Now using our result, we prove
a similar result (including the end points for the range) on
$\R^n$ where sets of Hausdorff dimension are replaced with the
sets of finite packing $\alpha-$measure. C. S Herz studied some
versions of $L^p$- Wiener Tauberian theorems and gave alternative
sufficient conditions for the translates of $f\in L^1\cap
L^p(\R^n)$ to span $L^p(\R^n)$ (See \cite{He}). With an additional
hypothesis on the zero sets of Fourier transform of $f$, we
improve his result. In \cite{RS}, Rawat and Sitaram initiated the
study of $L^p$-versions of the Wiener Tauberian theorem under the
action of motion group $M(n)$ on $\R^n$.
%Recall that $M(n)=\R^n\ltimes SO(n)$ and if $g\in M(n)$, $f$ a function on $\R^n$, $^gf$ will denote the function $^gf(x)=f(g.x)$.
We shall show that some of the
results proved in \cite{RS} can be improved using our result. Finally we
take up $L^p$-Wiener Tauberian theorem on the
Euclidean motion group $M(2)$.\\

In the remaining of this section we recall certain definitions
from Fractal geometry (See \cite{C} and \cite{P}). In the second
section we prove the above mentioned result on the
$L^p$-integrability and dimension of the support of $\wh{f}$ and
its sharpness. Finally in the third section we look at
applications of the results proved in section 2 to $L^p$-Wiener Tauberian theorems on $\R^n$ and the Euclidean motion group $M(2)$.\\

    Let $\CH_{\alpha}$ denote the Hausdorff $\alpha$-dimensional outer
measure. Let $E$ be a non-empty bounded subset of $\R^n$. The
\textbf{$\epsilon$-covering number} of $E$, $N(E,\epsilon)$, is
the smallest number of open balls of radius $\epsilon$ needed to
cover $E$. The \textbf{$\epsilon$-packing number} of $E$,
$P(E,\epsilon$), is the largest number of \textbf{disjoint} open
balls of radius $\epsilon$ with centres in $E$. The
\textbf{$\epsilon$-packing} of $E$ is any collection of disjoint
balls $\{B_{r_k}(x_k)\}_k$ with centres $x_k\in E$ and radii
satisfying $0<r_k\leq\epsilon/2$. Let $0\leq s<\infty$. For
$0<\epsilon<1$ and $A\subset\R^n$, put
$$P_{\epsilon}^s(A) = sup\{\sum_k(2r_k)^s\}$$
where the supremum is taken over all permissible
$\epsilon$-packings, $\{B_{r_k}(x_k)\}_k$ of $A$. Then
$P_{\epsilon}^s(A)$ is non-decreasing with respect to $\epsilon$
and we set the \textbf{packing pre measure}, $P^s_0$ as
$$P^s_0(A) = \underset{\epsilon\downarrow0}{lim}\ P_{\epsilon}^s(A).$$
We have $P^s_0(\emptyset)=0$, $P^s_0$ is monotonic and finitely
subadditive, but not countably subadditive. The \textbf{packing
$s-$ measure} of $A$, $\CP^s(A)$ is defined as
  \begin{equation}
    \CP^s(A)=inf\Big\{\sum_{i=1}^{\infty}P_0^s(A_i):A=\cup_{i=1}^{\infty}A_i\Big\}.\label{EqnIntro1}
  \end{equation}
Then $\CP^s$ is Borel regular (Theorem 3.11 in \cite{C}). If $\nu$
is a measure, the \textbf{$\alpha$-upper density of} $\nu$ at $x$,
$\ol{D^{\alpha}}(\nu,x)$ is defined as
$$\ol{D^{\alpha}}(\nu,x)\ =\ \underset{r\rightarrow0}{limsup}\ (2r)^{-\alpha}\nu(B_r(x)),$$
 where $B_r(x)$ is a ball of radius $r$ with centre $x$. Similarly \textbf{$\alpha$-lower density} of $\nu$ at x,
  $\ul{D^{\alpha}}(\nu,x)$ is defined using $liminf$. Let $\alpha<n$. A set $E\subset\R^n$ is said to be
\textbf{Ahlfors-David regular $\alpha$-set} if there exists $a,b$
(both $>0$) in $\mathbb{R}$ such that
$$0<ar^{\alpha}\leq\CH_{\alpha}(E\cap B_r(x))\leq
br^{\alpha}<\infty$$ for all $x\in E$ and $0<r\leq1$.
For all these definitions and similar ones we refer to \cite{C} and \cite{P}.\\

%%%%%%%%%%%%SECTION 2 $L^p$ Integrebility and supports of FT%%%%%%%%%%%%%%%%%%%%%%%%%%%%

\section{Dimensions of supports of Fourier transforms}
\setcounter{equation}{0}

In this section we relate the dimension of the support of the
Fourier transform of a function with its membership in $L^p$. In
the following lemma, we recall some needed results (see pages
78-89 in \cite{P}). For a non-empty subset $A$ of $\R^n$, let
$A(\epsilon)=\{x\in\R^n\ : d(x,A)<\epsilon\}$.

\begin{lem}\label{lemPM}
Fix $\epsilon>0$. Let $|A(\epsilon)|$ denote the Lebesgue measure of $A(\epsilon)$, where $A$
is a non-empty bounded subset of $\R^n$. Then,
\begin{enumerate}
\item $N(A,2\epsilon)\leq P(A,\epsilon) \leq N(A,\epsilon/2)$,\\
\item $\Omega_nP(A,\epsilon)\epsilon^n\leq |A(\epsilon)|\leq
\Omega_nN(A,\epsilon)(2\epsilon)^n$, where $\Omega_n$ denotes the volume of the unit ball in $\R^n$,\\
\item For $0\leq s<\infty,\ P(A,\epsilon/2)\epsilon^{s}\leq P_{\epsilon}^{s}(A)$,\\
\item Let $B\subset \R^n$ be such that $\CH_{\alpha}(B)<\infty$. Then
 $$2^{-\alpha}\leq\overline{D^{\alpha}}(\mu,x)\leq1$$
 for $\CH_{\alpha}$ almost all $x\in B$, where $\mu=\CH_{\alpha}|_B$.
\end{enumerate}
\end{lem}
The following lemma is crucial for us.

%%%%%%%%%%MAIN LEMMA%%%%%%%%%%%%%%%%%%%%%%%%%%%%%%%%%%%%%%%%%%%%%%%%%%%%%%%%%

\begin{lem}\label{lemMY}
%Let $E$ be an $\alpha$-type set in $\R^n$
Let $0\leq\alpha<n$ and let $E\subset\R^n$ be such that $\CP^{\alpha}(E)<\infty.$ Let $S\subset E$ be bounded and
$S(\epsilon)=\{x\in\R^n: d(x,S)<\epsilon\}$. Then
$$\underset{\epsilon\rightarrow0}{limsup}\ |S(\epsilon)|\epsilon^{\alpha-n}<\infty,$$
 where
$|S(\epsilon)|$ denotes the Lebesgue measure of $S(\epsilon)$.
\end{lem}
\begin{proof}
Since we have $\CP^{\alpha}(S)\leq\CP^{\alpha}(E)<\infty$, there
exists a countable cover $\{\wt{A_i}\}$ of $S$ such that $\sum
P_0^{\alpha}(\wt{A_i})<\infty$. Let $R>0$ be such that $S\subset
B_R(0)$. Then $\{A_i\}$ also covers $S$, where $A_i=\wt{A_i}\cap
B_R(0)$ is bounded and $\sum P_0^{\alpha}(A_i)\leq\sum
P_0^{\alpha}(\wt{A_i})<\infty$. By Lemma \ref{lemPM}, \Bea
|A_i(\epsilon)| &\leq& \Omega_n(2\epsilon)^{n}N(A_i,\epsilon)\\
                   &\leq& \Omega_n(2\epsilon)^nP(A_i,\epsilon/2)\\
                   &\leq& \Omega_n2^n\epsilon^{n-\alpha}P_{\epsilon}^{\alpha}(A_i).
\Eea Hence $\epsilon^{\alpha-n}|A_i(\epsilon)| \leq
C_{n}P^{\alpha}_{\epsilon}(A_i)$ for some fixed constant $C_n$. We
also have $|S(\epsilon)|\leq\sum|A_i(\epsilon)|$. Hence,
$\epsilon^{\alpha-n}|S(\epsilon)|\leq C_n\sum
P_{\epsilon}^{\alpha}(A_i).$ So, \bee
\underset{\epsilon\rightarrow0}{limsup}\
\epsilon^{\alpha-n}|S(\epsilon)| \leq C_n\sum P_0^{\alpha}(A_i) <
  \infty.
  \eee
Hence we have $\epsilon^{\alpha-n}|S(\epsilon)|$ tending to a finite limit as
  $\epsilon\rightarrow0$.
\end{proof}

%%%%%%%%%%%%%%%%%%UNIQUENESS THEOREM%%%%%%%%%%%%%%%%%%%%%%%%%%%%%%%%%%%%%%%%%%%%%%

\begin{thm}\label{thmMY}
Let $f\in L^{p}(\mathbb{R}^n)$ be such that $supp\ \wh{f}$ is contained in a set $E$ where
$\CP^{\alpha}(E)<\infty$ for some $0\leq\alpha< n$. Then $f\equiv0$, provided $p\leq\frac{2n}{\alpha}$.\\
\end{thm}
\begin{proof} For the proof we closely follow the arguments in
\cite{HA} (See page 174 of \cite{H}). By convolving $f$ with a
compactly supported smooth function we can assume that $f\in
L^{p}(\R^{n})$ where $p=2n/\alpha$. Choose an even function $ \chi
\in C_{c}^{\infty} (R^{n}) $ with support in unit ball and
$\int_{\mathbb{R}^{n}}\chi(x) dx = 1$. Let $\chi_{\epsilon}(x) =
\epsilon^{-n}\chi(x/\epsilon)$ and
$u_{\epsilon}=u\ast\chi_{\epsilon}$ where $u=\wh{f}$. Then by the
Plancherel theorem,
  \Bea
    \|u_{\epsilon}\|^{2} &=& \int_{\mathbb{R}^{n}}|f(x)|^{2}|\wh{\chi}(\epsilon x)|^{2} dx \\
     &\leqslant&  C\epsilon^{\alpha-n} \sum_{j=-\infty}^{\infty} 2^{j(n-\alpha)} \sup_{2^{j}\leqslant|\epsilon x|\leqslant2^{j+1}}|\wh{\chi}_{\epsilon}(x)|^2(2^{-j}\epsilon)^{n-\alpha}\underset{2^{j}\leqslant|\epsilon x|\leqslant2^{j+1}}{\int}|f(x)|^{2}dx \\
      &=&C\epsilon^{\alpha-n}\sum_{j=-\infty}^{\infty}
      a_{j}b_{j}^{\epsilon},
  \Eea
  where
  \bee
    a_{j}=2^{j(n-\alpha)} \sup_{2^{j}\leqslant|x|\leqslant2^{j+1}}
    |\wh{\chi}(x)|^{2},
  \eee
  and
  \bee
    b_{j}^{\epsilon}=(2^{-j}\epsilon)^{n-\alpha}\int_{2^{j}\leqslant|\epsilon x|\leqslant2^{j+1}}|f(x)|^{2}dx.
  \eee
  Applying Holder's inequality,
  \begin{equation*}
    |b_{j}^{\epsilon}|\leqslant
    C\left(\int_{2^{j}\epsilon^{-1}\leqslant|x|\leqslant2^{j+1}\epsilon^{-1}}|f(x)|^{p}dx\right)^{2/p},
  \end{equation*}
  which goes to zero as $\epsilon \rightarrow 0$, for any fixed j. Also we have $|b^{\epsilon}_{j}|\leqslant C\|f\|^{2}_{p}<\infty$ for some constant $C$ independent of $\epsilon$ and $j$. Since $ \sum_j|a_{j}| $ is finite, by the dominated convergence theorem, we have $\sum_ja_{j}b_{j}^{\epsilon} \rightarrow 0$ as $ \epsilon \rightarrow 0$.\\

  Let $\psi\in C_c^{\infty}(\mathbb{R}^n)$. Let $S=supp\ \wh{f}\cap supp\ \psi$. Then $S$ is a bounded subset of $E$. By Lemma \ref{lemMY}, we have $\epsilon^{\alpha-n}|S_{\epsilon}|$ tending to a finite limit as
  $\epsilon\rightarrow0$. So,
  \begin{eqnarray*}
    |<u,\psi>|^2 &=& \underset{\epsilon\rightarrow0}{lim}\ |<u_{\epsilon},\psi>|^2 \\
      &\leq& \underset{\epsilon\rightarrow0}{lim}\ \|u_{\epsilon}\|_2^2\int_{S_{\epsilon}}|\psi|^2 \\
      &\leq& c\|\psi\|_{\infty}^2\ \underset{\epsilon\rightarrow0}{lim}\ \epsilon^{\alpha-n}|S_{\epsilon}|\sum_{j=-\infty}^{\infty} a_{j}b_{j}^{\epsilon} \\
      &=&0
  \end{eqnarray*}
 Hence $f=0$.   \\
\end{proof}

%%%%%%%%%%%%%%%%%%%%%%%%%%%%%%%%%%%%%%%%%%%%%%%%%%%%%%%%%%%%%%%%%%%%%%%%%%%%%%%%%%%%%

\begin{rem}
%Let $\alpha$ be an integer, and $M$ be any $\mathcal{C}^1$ submanifold of $\R^n$. Then the restriction of Hausdorff measure $\mu_{\alpha}$ to $M$ gives a %constant multiple of the surface measure on $M$ (See page 56 in \cite{P}).
For an integer $0\leq d<n$, any $d$-dimensional smooth manifold in
$\R^n$ has both Hausdorff and Packing dimension as $d$. (See page
56 and 85 in \cite{P}) Hence the above Theorem \ref{thmMY} extends Theorem
1 in \cite{NA}.
\end{rem}

\begin{lem}\label{lemC}
Let $E\subset\R^n$ be such that $0<\CH_{\alpha}(E)<\infty$. Assume that
there exists constants $0<a<\infty$ and $r_{\alpha}>0$ such that
$ar^{\alpha}\leq\CH_{\alpha}(E\cap B_r(x))$ for all $x\in E$ and
for all $r<r_{\alpha}$. Then $\CP^{\alpha}(E)<\infty$.
\end{lem}

\begin{proof}
Consider $\mu_{\alpha}(A)=\CH_{\alpha}(E\cap A)$ for all
$A\subseteq\R^n$. Then $\mu_{\alpha}$ is a finite Borel regular
measure on $\R^n$. Since  $ar^{\alpha}\leq\CH_{\alpha}(E\cap
B_r(x))$ for all $x\in E$ and for all $r<r_{\alpha}$, we have
$a\leq\ul{D^{\alpha}}(\mu_{\alpha},x)$ for all $x\in E$. Let
$k=inf_{x\in E}\{\ul{D^{\alpha}}(\mu_{\alpha},x)\}$. We have
$0<a\leq k$. By Lemma \ref{lemPM} (4), we have $k<\infty$. Also
corresponding to each $E\subseteq \R^n$ is a Borel set $B\supseteq
E$ such that $\CP^{\alpha}(B)= \CP^{\alpha}(E)$. (See Theorem
3.11(c) in \cite{C}). Then, $\mu_{\alpha}(B)=\CH_{\alpha}(E\cap
B)<\infty$. By Theorem 3.16 in \cite{C}, we have
$\CP^{\alpha}(E)=\CP^{\alpha}(B)\leq\mu_{\alpha}(B)/k<\infty$.
\end{proof}

\begin{rem}\label{RemAhlfors}
Ahlfors-David regular $\alpha$-sets satisfy the hypothesis of the
above lemma.
\end{rem}

%%%%%%%%%%%%%%%%%%%%% SHARPNESS OF THE UNIQUENESS THEOREM%%%%%%%%%%%%%%%%%%%%%%%%%%%%%%%%%

Next we show that Theorem \ref{thmMY} is sharp. First, let us
recall a well known example due to Salem which shows that there
exists a measure $\nu$ supported on a Cantor type set
$K\subseteq\R$, of Hausdorff dimension $\beta,\ 0<\beta<1$ with
Fourier tranform $\wh{\nu}$ belonging to $L^q(\R)$ for all
$q>2/\beta$ (See \cite{Sa} and page 263-271 in \cite{W}). Let
$M=K\times K\times...\times K$ ($n$ times) and
$\mu=\nu\times\nu\times...\times\nu$ ($n$ times). Then $\mu$ is
supported in $M$ and $\wh{\mu}\in L^q(\R^n)$ for
$q>\frac{2}{\beta}=\frac{2n}{\alpha}$ where
 $\alpha=n\beta$. Closely following the proof in \cite{W} (page 33)
 we show that not only the Hausdorff dimension of $M$ is $\alpha$,
 but $M$ also satisfies the hypothesis of the above Lemma \ref{lemC}
 and thus proving that the range in Theorem \ref{thmMY} is the best possible.\\

First, we briefly recall how the above set $K\subseteq\R$ is
constructed. Choose a positive number $\eta$ and an integer $N$ so
that $N\eta<1$ and \begin{equation}
N\eta^{\beta}=1.\label{Eqn25LB2}
\end{equation} Choose $N$
independent points $a_i$ in the unit interval $[0,1]$ in such a
way that $0\leq a_1<a_2<...<a_N\leq1-\eta$ and widely enough
spaced so that the distance between two $a_i$ is larger than
$\eta$. The set $K$ is constructed as the intersection of
decreasing sequence of compact
sets $K_j$, where $K_j$'s are defined as follows:\\

 Choose an increasing sequence of non-zero positive
numbers $\eta_j$ converging to $\eta$ where
\begin{equation}
\eta(1-\frac{1}{(j+1)^2})\leq\eta_j\leq\eta \label{Eqn25LB1}
\end{equation} for all $j$. The first set, $K_1$,
is the union of $N$ intervals of length $\eta_1$ of the form
$[a_k,a_k+\eta_1]$. The second set $K_2$, has $N^2$ intervals of
length $\eta_1\eta_2$ of the form
$[a_i+a_j\eta_1,a_i+a_j\eta_1+\eta_1\eta_2]$ and so on.
Inductively, we obtain a sequence $K_j$ of decreasing sets of
length $\eta_1\eta_2...\eta_j$. Then $K=\cap_{j}K_j$. It
is known that the Hausdorff dimension of $K$ is $\beta$. (see \cite{Sa} and page 268 in \cite{W})\\

\begin{lem}\label{lemEX}
Hausdorff dimension of $M=K\times K\times..\times K$ ($n$ times)
equals $\alpha=n\beta$ and $0\leq\CP^{\alpha}(M)<\infty$.
\end{lem}
\begin{proof}
Let $M_j=K_{j_1}\times K_{j_2}\times...\times K_{j_n}$ for
$j=(j_1,j_2,...,j_n)$. Among the coverings of $M$ which compete in
the definition of $\CH_{\alpha}(M)$, (the Hausdorff measure of
$M$) are the coverings $M_j$ themselves, consisting of
$N^{j_1+j_2+...+j_n}$ cubes of volume
$\Pi_{i=1}^n(\eta_1\eta_2...\eta_{j_i})$. Since
$\eta_1\leq\eta_2\leq...<\eta$, we obtain $$ \CH_{\alpha}(M) \leq
N^{j_1}(\eta_1^{\beta}\eta_2^{\beta}...\eta_{j_1}^{\beta})...N^{j_n}(\eta_1^{\beta}\eta_2^{\beta}...\eta_{j_n}^{\beta})
\leq N^{j_1}\eta^{j_1\beta}...N^{j_n}\eta^{j_n\beta} = 1,$$
and see that the dimension of $M$ is at most $\alpha$.\\

To show that the dimension of $M$ is exactly $\alpha$, we show
that $\CH_{\alpha}(M)$ is not $0$. First we prove that $M$
satisfies the hypothesis of Lemma \ref{lemC}.\\

Let $0<r<1$ and $x\in M$, that is let $x=(x_1,x_2,...,x_n)$ where
$x_m\in K$ for all $m$. For every $m$, by construction of $K$,
there exists a smallest integer $t_m$ such that $K\cap (x_m-r,
x_m+r)$ contains at least one interval $I_{t_m}$ of length
$\eta_1..\eta_m$. Thus
\begin{equation} K\cap (x_m-r, x_m+r) \supseteq K\cap
I_{t_1}\times...\times K\cap I_{t_n}. \label{Eqn25L6}
\end{equation}

Since Hausdorff measure is translation invariant, we can assume
$2r\leq\eta_1...\eta_{t_m-1}$. Since $\alpha=n\beta,$
\begin{equation}
(2r)^{\alpha} \leq
\Pi_{m=1}^n(\eta_1^{\beta}...\eta_{t_m-1}^{\beta}).\label{Eqn25L1}
\end{equation}

 In computing the Hausdorff measure, it is
enough to take the infimum of $\Sigma d_i^{\alpha}$ over all
coverings of $M\cap B_r(x)$ by countable families of (sufficiently
small) open balls $A_i$, where the end points of the projection of
$A_i$ to $m^{th}$ axis is in the complement of $K\cap
(x_m-r,x_m+r)$. From the compactness, it is also clear that these
coverings consist of only a finite number of disjoint, open cubes.
Let $\{U_i\}$ be one such family of sufficiently small cubes that
cover $M\cap B_r(x)$, where the end points of the projection of
$U_i$ to $m^{th}$ axis is in the complement of $K\cap
(x_m-r,x_m+r)$.\\

 Let $p_{i_m}$ be the smallest integer $p$ such that $m^{th}$
projection of $U_i$ contains at least one interval of $K_p$ and
$P_i=(p_{i_1},p_{i_2},...p_{i_n})$. Then, from (\ref{Eqn25L6}),
$t_m\leq p_{i_m}$. Let
\begin{equation}
p_{i_m}=t_m+s_{i_m}\label{Eqn25L2}
\end{equation}
 Let $m^{th}$ projection of $U_i$
contain $k_i^{(m)}$ number of constituent intervals of $K_{p_m}$.
Then $U_i$ contain $k_i=\Pi_{m=1}^nk_i^{(m)}$ number of cubes of
$M_{P_i}=K_{p_{i_1}}\times...\times K_{p_{i_n}}$. Let $d_i$ denote
the diameter of $U_i$. Then
\begin{equation}
d_i^n\geq
k_i\Pi_{m=1}^n(\eta_1\eta_2...\eta_{p_{i_m}}).\label{Eqn25L5}
\end{equation}
 Let $j_m$'s be large such
that $\cup U_i$ contains $M_j\cap M\cap B_r(x)$ where
$M_j=K_{j_1}\times...\times K_{j_n}$ and $M_j\subset M_{P_i}$, for
all $i$. Then $U_i$ contains
$k_iN^{(j_1-p_{i_1}+...+j_n-p_{i_n})}$ cubes of $M_j$. By
(\ref{Eqn25L2}), $U_i$ contains
$k_iN^{(j_1-t_1-s_{i_1}+...+j_n-t_n-s_{i_n})}$ cubes of $M_j$.
However by (\ref{Eqn25L6}), $$M\cap B_r(x)\cap M_j\subseteq
M_t\subset M\cap B_r(x),$$ where $M_t=(K\cap
I_{t_1})\times...\times(K\cap I_{t_n})$. So the number of cubes of
$M_j$ covered by $\cup U_i$ is at least $N^{j_1-t_1+...+j_n-t_n}$.
Since $\sum_ik_iN^{(j_1-t_1-s_{i_1}+...+j_n-t_n-s_{i_n})}$ is the
total number of cubes of $M_j$ covered by $\cup U_i$,
\begin{equation}
\sum_ik_iN^{(j_1-t_1-s_{i_1}+...+j_n-t_n-s_{i_n})}\geq
N^{j_1-t_1+...+j_n-t_n}.\label{Eqn25L7}
\end{equation}
The equation (\ref{Eqn25L5}) implies that
\begin{eqnarray*}
d_i^{\alpha} &\geq&(k_i\Pi_{m=1}^n(\eta_1\eta_2...\eta_{p_{i_m}}))^{\beta}\\
                       &\geq& (2r)^{\alpha}(k_i\Pi_{m=1}^n(\eta_{t_m}\eta_{t_m+1}...\eta_{p_{i_m}}))^{\beta}(\text{from}\
                       (\text{\ref{Eqn25L1}}))\\
                       &\geq&(2r)^{\alpha}(k_i\Pi_{m=1}^n\eta_{t_m}\eta^{p_{i_m}-t_m}[(1-\frac{1}{(t_m+1)^2})...(1-\frac{1}{p_{i_m}^2})])^{\beta}\
                       \text{from (\ref{Eqn25LB1})}
\end{eqnarray*}
Since $\eta_m$ is an increasing sequence and by (\ref{Eqn25LB1}),
$\eta_{t_1}\eta_{t_2}...\eta_{t_n}\geq(\frac{3}{4}\eta)^n$. Fix
$C=(\frac{3}{4}\eta)^n$. Thus
\begin{eqnarray*}
d_i^{\alpha}        &\geq& C(2r)^{\alpha}\big(k_i\eta^{(p_{i_1}+...+p_{i_n}-(t_1+...t_n))}\Pi_{m=1}^n\big[(1-\frac{1}{t_m+1})(1+\frac{1}{p_{i_m}})\big]\big)^{\beta}\\
                       &\geq& C(2r)^{\alpha}\big(k_i\eta^{(p_{i_1}+...+p_{i_n}-(t_1+...t_n))}\Pi_{m=1}^n\big[\frac{1}{2}(1+\frac{1}{p_{i_m}})\big]\big)^{\beta}\\
                        &>& Cr^{\alpha}k_i^{\beta}\eta^{(p_{i_1}+...+p_{i_n}-(t_1+...t_n))\beta},
\end{eqnarray*}
From (\ref{Eqn25LB2}), we have
$$N^{(j_1+...j_n)-(p_{i_1}+...p_{i_n})}\eta^{(j_1+...j_n-t_1-...t_n)\beta}=\eta^{(p_{i_1}+...+p_{i_n}-(t_1+...t_n))}.$$
Thus
\begin{equation}
  d_i^{\alpha}\geq
                Cr^{\alpha}k_i^{\beta}N^{(j_1+...j_n)-(p_{i_1}+...p_{i_n})}\eta^{(j_1+...j_n-t_1-...t_n)\beta}.\label{Eqn25L3}
\end{equation}
Also, there exists a constant $C_{N,n}$ ($=2^n(N-1)^n$), such that
$1\leq k_i\leq C_{N,n}$ because of the choice of $p_{i_k}$. Let
$L=(C_{N,n})^{\beta-1}$. Since $0<\beta<1$,
\begin{equation}
k_i^{\beta}>Lk_i\label{Eqn25L4}
\end{equation}

 From (\ref{Eqn25L2}) and (\ref{Eqn25L4}),
summing over $i$ in (\ref{Eqn25L3}), we have
\begin{eqnarray*}
  \Sigma_i d_i^{\alpha} &\geq& CLr^{\alpha}\eta^{(j_1+...j_n-t_1-...-t_n)\beta}\Sigma_ik_iN^{(j_1+...j_n)-(t_{1}+...t_{n}+s_{i_1}+...s_{i_n})}\\
                       &\geq& CLr^{\alpha}\eta^{(j_1+...j_n-t_1-...-t_n)\beta}N^{j_1+...j_n-t_1-...-t_n}\ (\text{from}\ (\ref{Eqn25L7})) \\
                       &=& CLr^{\alpha}\ \ (\text{from}\ (\ref{Eqn25LB2})) \\
                       &>& 0
  \end{eqnarray*}
  Thus $\CH_{\alpha}(M\cap B_r(x))\geq CLr^{\alpha}$ for all $x\in M$ and $0<r<1$. Similarly we prove that $\CH_{\alpha}(M)>0$. By Lemma \ref{lemC}, $\CP^{\alpha}(M)<\infty$.
\end{proof}

\begin{rem}
As remarked by one of the referees of this paper, the set
constructed in Lemma \ref{lemEX} is fractal even if $\alpha$ is an
integer. In \cite{NA}, the authors proved the sharpness of Theorem
2 (in \cite{NA}) for any integer $\alpha \geq n/2$ by constructing
a smooth manifold $M \subset \mathbb R^n$ and $\mu$ supported on
$M$ such that the Fourier transform $f=\hat{\mu}\in
L^p(\mathbb{R}^n)$ for all $p > 2n/\alpha$. It would be
interesting to see if this can be done for all integers $\alpha$
between $0$ and $n$.
\end{rem}

%%%%%%%%%%%%% SECTION 3 APPLICATIONS%%%%%%%%%%%%%%%%%%%%%%%%%%%%%%%%%%%%%%%%
\section{Applications to Wiener Tauberian Theorems}
\subsection{ $L^p$ Wiener Tauberian Theorems on $\R^n$}
\setcounter{equation}{0}

    In this section, we improve the results on $L^p$ versions of
Wiener Tauberian type theorems on $\R^n$ obtained in \cite{RS}.
Consider the motion group $M(n)=\R^n\ltimes SO(n)$ with the group
law
$$(x_1,k_1)(x_2,k_2)=(x_1+k_1x_2,k_1k_2).$$
For a function $h$ on $\R^n$ and an arbitrary $g=(y,k)\in M(n)$,
let $^gh$ be the function $^gh(x)=h(kx+y),\ x\in\R^n.$ Let
$\wh{h}$ denote the Euclidean Fourier transform of the function
$h$. For $h\in L^1\cap L^p(\R^n),\ 1\leq p\leq\infty$, let
$S=\{r>0:\wh{h}\equiv0$ on $C_r\}$, where $C_r$ is the sphere of
radius $r>0$ centered at origin in $\R^n$. Let $Y\ =\
Span\{^gh:g\in M(n)\}$. Then the main result from \cite{RS} is

\begin{thm}\label{thmRS}
\begin{enumerate}
\item If $p=1$, then $Y$ is dense in $L^1(\R^n)$ if and only if
$S$ is empty and $\wh{h}(0)\neq0.$ \item If $1<p<\frac{2n}{n+1}$,
then Y is dense in $L^p(\R^n)$ if and only if $S$ is empty. \item
If $\frac{2n}{n+1}\leq p<2$, and every point of $S$ is an isolated
point, then $Y$ is dense in $L^p(\R^n)$. \item If $2\leq
p\leq\frac{2n}{n-1}$, and $S$ is of zero measure in $\R^+$, then
$Y$ is dense in $L^p(\R^n)$. \item If $\frac{2n}{n-1}<p<\infty$,
then $Y$ is dense in $L^p(\R^n)$ if and only if $S$ is nowhere
dense.
\end{enumerate}
\end{thm}
We prove that the part (3) of the above theorem can be improved:
\begin{thm}\label{thmMYRr}
Let $f\in L^1(\mathbb{R}^n)\cap L^p(\mathbb{R}^n)$ and let
$S=\{r>0:\wh{f}\equiv0$ on $C_r\}$ be such that
$\CP^{\beta}(S)<\infty,$ for some $0\leq\beta<1$. If
$\frac{2n}{n+1-\beta}\leq p\leq 2$, then $Y\ =\ Span\{^gf:g\in
M(n)\}$ is dense in $L^p(\mathbb{R}^n)$.
\end{thm}
\begin{proof}
    Fix $\epsilon<1$. Suppose $Y$ is not dense in $L^p(\mathbb{R}^n)$. Let $h\in
L^q(\mathbb{R}^n)$ annihilate all the elements in $Y$, where
$\frac{1}{p}+\frac{1}{q}=1$. We can assume $h$ to be smooth,
bounded and radial (see the arguments in \cite{RS}). It follows
that $h*f\equiv0$. Then $supp$ $\wh{h}$ is contained in the zero
set of $\wh{f}$. Let $\alpha$ be such that $2\leq
q=\frac{2n}{\alpha}\leq\frac{2n}{n-1+\beta}$. Choose an even
function $\chi \in C_{c}^{\infty}(\R^{n}) $ with support in the
unit ball and $\int_{\mathbb{R}^{n}}\chi(x) dx = 1$. Let
$\chi_{\epsilon}(x) = \epsilon^{-n}\chi(x/\epsilon)$ and
$u_{\epsilon}=u\ast\chi_{\epsilon}$ where $u=\wh{h}$. Since $2\leq
q$, as in Theorem \ref{thmMY},
  \begin{eqnarray*}
    \|u_{\epsilon}\|^{2} &\leqslant& C\epsilon^{\alpha-n}\sum_{j=-\infty}^{\infty}
    a_{j}b_{j}^{\epsilon},\\
    \text{where}\ a_{j} &=& 2^{j(n-\alpha)} \sup_{2^{j}\leqslant|x|\leqslant2^{j+1}}
    |\wh{\chi}(x)|^{2}\\
    \text{and}\ b_{j}^{\epsilon} &=& (2^{-j}\epsilon)^{n-\alpha}\int_{2^{j}\leqslant|\epsilon x|\leqslant2^{j+1}}|h(x)|^{2}dx.
  \end{eqnarray*}
  and $\sum\limits_{j=-\infty}^{\infty}a_{j}b_{j}^{\epsilon} \rightarrow 0$ as $ \epsilon \rightarrow 0$.\\

  Let $\psi\in C_c^{\infty}(\mathbb{R}^n)$. Let $M=\ supp\ \wh{h}\ \cap\ supp\ \psi$ and
  let $R_{\psi}>0$ be such that $M$ is contained in a ball of radius $R_{\psi}$.
  For $x\in M$, $\|x\|\in S$ and $\|x\|\leq R_{\psi}$. Let $S_{\psi}=\{r\in S: r\leq R_{\psi}\}$.
  Then $S_{\psi}$ is a bounded subset of $S$. With similar arguments in Lemma \ref{lemMY}, we prove that
  $\underset{\epsilon\rightarrow0}{lim}\ \epsilon^{\beta-1}\int_{M_{\epsilon}}|\psi(x)|^2dx<\infty$:\\

  Since $\CP^{\beta}(S_{\psi})\leq\CP^{\beta}(S)<\infty$,
  let $\{A_i\}$ be a cover of $S_{\psi}$ such that $\sum_iP_0^{\beta}(A_i)<\infty$. Then $P_0^{\beta}(A_i\cap S_{\psi})<\infty$.
  For $S^i_{\psi}=A_i\cap S_{\psi}$, let $P(S^i_{\psi},\epsilon)$ be the maximum number of disjoint balls with centers $\{r_j\}$ in $S^i_{\psi}$, of
  radius $\epsilon$ and $N(S^i_{\psi},\epsilon)$ be the $\epsilon$-covering number of $S^i_{\epsilon}$. Then
  $$S^i_{\psi}\subseteq\cup_{j=1}^{N(S^i_{\psi},\epsilon)}(r_j-\epsilon/2,r_j+\epsilon/2)\
  \text{and}$$
  $$S_{\psi}(\epsilon)\subset\cup_iS^i_{\psi}(\epsilon)\subseteq\cup_i\cup_{j=1}^{N(S^i_{\psi},\epsilon)}(r_j-\epsilon,r_j+\epsilon).$$
  If $x\in M(\epsilon)$, then $\|x\|\in S_{\psi}(\epsilon)$. We have,
  \begin{eqnarray*}
    \int_{M(\epsilon)}|\psi(x)|^2dx &\leq & \int_{r\in S_{\psi}(\epsilon)}\int|\psi(r\omega)|^2d\omega r^{n-1} dr \\
      &\leq & (R_{\psi}+\epsilon)^{n-1} \int_{r\in S_{\psi}(\epsilon)}\int|\psi(r\omega)|^2d\omega dr \\
      &\leq & (R_{\psi}+1)^{n-1} \|\psi\|^2_{\infty}\Omega_n \sum_i\sum_{j=1}^{N(S^i_{\psi},\epsilon)}\int_{r_j-\epsilon}^{r_j+\epsilon} dr  \\
      &=& C_1\sum_iN(S^i_{\psi},\epsilon)(2\epsilon) \\
      &\leq & 2C_1 \epsilon \sum_iP(S^i_{\psi},\epsilon/2)\ \ \ \text{(by lemma
      \ref{lemPM})}
  \end{eqnarray*}
  where $C_1=(R_{\psi}+1)^{n-1} \|\psi\|^2_{\infty}\Omega_n $ is a constant independent of $\epsilon$ and $\Omega_n$ is the volume
  of the unit sphere in $\mathbb{R}^n$. Thus,
  $$\underset{\epsilon\rightarrow0}{lim}\ \epsilon^{\beta-1}\int_{M(\epsilon)}|\psi(x)|^2dx\leq
  2C_1 \sum_i\underset{\epsilon\rightarrow0}{lim}\ \epsilon^{\beta} P(S^i_{\psi},\epsilon/2)\leq 2C_1\sum_iP_0^{\beta}(A_i)<\infty.$$
  Hence,
  \begin{eqnarray*}
    |<u,\psi>|^2 &=& \underset{\epsilon\rightarrow0}{lim}\ |<u_{\epsilon},\psi>|^2 \\
      &\leq& \underset{\epsilon\rightarrow0}{lim}\ \|u_{\epsilon}\|_2^2\ \int_{M_{\epsilon}}|\psi|^2 \\
      &\leq& C\underset{\epsilon\rightarrow0}{lim}\ \epsilon^{\alpha-n}\sum_{j=-\infty}^{\infty} a_{j}b_{j}^{\epsilon}\int_{M_{\epsilon}}|\psi(x)|^2dx \\
      &\leq& C\underset{\epsilon\rightarrow0}{lim}\ \epsilon^{\alpha-n-\beta+1}\epsilon^{\beta-1}\sum_{j=-\infty}^{\infty} a_{j}b_{j}^{\epsilon}\int_{M_{\epsilon}}|\psi(x)|^2dx \\
      &\leq& CC_1\underset{\epsilon\rightarrow0}{lim}\ \sum_{j=-\infty}^{\infty} a_{j}b_{j}^{\epsilon}\\
      &=&0,
  \end{eqnarray*}
  since $2\leq\frac{2n}{\alpha}\leq\frac{2n}{n-1+\beta}$, that is $0\leq\alpha-n-\beta+1$. Hence $h=0$.
\end{proof}

\begin{rem}
Suppose every point of $S$ is an isolated point. Convolving $f$
with an arbitrary Schwartz class function whose Fourier transform
is compactly supported, we may assume that $S$ is finite. The case
$\beta=0$ in the above theorem then implies part (3) of Theorem
\ref{thmRS}.
\end{rem}

Now let $f$ be an integrable function in $L^1\cap L^p(\R)$ and let
$F$ denote the closed set where the Fourier transform of $f$
vanishes. In \cite{B}, A. Beurling proved that if for some $p$ in
$(1,2)$, the space of finite linear combinations of translates of
$f$ is not dense in $L^p(\R)$, then the Hausdorff dimension of $F$
is at least $2-(2/p)$ (see also page 312 in \cite{W}). In other
words, if the Hausdorff dimension of $F$ is $\alpha$, for
$0\leq\alpha\leq1$, then the space of finite linear combinations
of translates of $f$ is dense in $L^p(\R)$ for
$2/(2-\alpha)<p<\infty$. Now using Theorem \ref{thmMY}, we prove a
similar result (including the end points for the range) on $\R^n$
where Hausdorff dimension is replaced with the packing dimension.

\begin{thm}\label{thmMYRB}
Let $f\in L^1(\mathbb{R}^n)\cap L^p(\mathbb{R}^n)$ for
$\frac{2n}{2n-\alpha}\leq p<\infty$ and let the zero set of
$\wh{f}\subseteq E,$ where $\CP^{\alpha}(E)<\infty$ for some
$0\leq\alpha<n$. Then $X=span\{ ^xf: x\in\mathbb{R}^n\}$ is dense in
$L^p(\mathbb{R}^n)$.
\end{thm}
\begin{proof}
Suppose $X$ is not dense in $L^p(\R^n)$. Then there exists a non
trivial, smooth and radial $h\in L^q(\mathbb{R}^n)$ such that $h *
f_1 \equiv 0$ for all $f_1\in X$(see the arguments in \cite{RS}).
Clearly the zero set of $X(\subset L^1(\mathbb{R}^n))$,
$\underset{u\in X}{\cap}\{s\in\mathbb{R}^n:\wh{u}(s)=0\}$ is equal
to the zero set of $\wh{f}$, $Z(\wh{f})$. Hence $supp\
\wh{h}\subseteq Z(\wh{f})$. Since $\frac{2n}{2n-\alpha}\leq
p<\infty$, we have $1< q\leq \frac{2n}{\alpha}$.
By Theorem \ref{thmMY}, $h=0$. Thus $X$ is dense in $L^p(\mathbb{R}^n)$.\\
\end{proof}

In \cite{He}, C. S Herz studied the versions of $L^p$- Wiener
Tauberian theorems. From Theorem 1 and Theorem 4 of \cite{He}, we
note that for $f\in L^1\cap L^p(\R^n),\ p<\infty$ the alternative
sufficient conditions for the translates of $f$ to span $L^p$ are,
\begin{enumerate}
\item $|K(\epsilon)|=o(\epsilon^{n(1-2/q)})$ for each compact
subset $K$ of $E$. \item dim $E=\alpha<2n/q$, with the proviso, if
$n>2$, that $q\leq 2n/(n-2)$.
\end{enumerate} where $E$ denotes the zero set of $\wh{f}$ and $\frac{1}{p}+\frac{1}{q}=1$.
With an additional hypothesis on $E$, using Theorem \ref{thmMYRB}, we can
improve the result in \cite{He}:

\begin{prop} For $f\in
L^1\cap L^p(\R^n),\ 1\leq p<\infty$ a sufficient condition that
the translates of $f$ span $L^p$ is : the zero set of $\wh{f}$ has
finite packing $\alpha$- measure for $\alpha\leq 2n/q$ where
$\frac{1}{p}+\frac{1}{q}=1$.
\end{prop}

\subsection{ $L^p$ Wiener Tauberian Theorem on $M(2)$}
\setcounter{equation}{0}

In this section, we look at one sided and two sided analogues of
Wiener Tauberian Theorems on $M(2)$ and improve a few
results from \cite{NR}.\\

The group $M(2)$ is the semi-direct product of $\C$ with the
special orthogonal group $K=SO(2)$. The group law in $G=M(2)$ is
given by \bee (z,e^{i\alpha})(w,e^{i\beta})=(z+e^{i\alpha}w,
e^{i(\alpha+\beta)}). \eee The Haar measure on $G$ is given by
$dg=dzd\alpha$ where $dz$ is the Lebesgue measure on $\C$ and
$d\alpha$ is the normalized Haar measure on $S^1$. For each
$\lambda>0$, we have a unitary irreducible representation of $G$
realized on $H=L^2(K)=L^2([0,2\pi],dt)$, given by
$$[\pi_{\lambda}(z,e^{it})u](s)=e^{i\lambda<z,e^{is}>}u(s-t),$$
for $(z,e^{it})\in G$ and $u\in H$. Here $<z,w>=$Re$z.\bar{w}$. It
is known that these are all the infinite dimensional, non
equivalent unitary irreducible representations of $G$. Apart from
the above family, we have another family $\{\chi_n,n\in \Z\}$,
where $\Z$ is the set of integers, of one dimensional unitary
representations of $G$, given by
$\chi_n(z,e^{i\alpha})=e^{in\alpha}$. Then the unitary dual
$\wh{G}$, of $G$ is the collection
$\{\pi_{\lambda},\lambda>0\}\cup\{\chi_n:n\in\Z\}$ (see page 165,
\cite{S}).\\

For $f\in L^1(G)$, define the "group theoretic" Fourier transform
of $f$ as follows:
$$\pi_{\lambda}(f)=\int_Gf(g)\pi_{\lambda}(g)dg,\ \lambda>0$$ and
$$\chi_n(f)=\int_Gf(z,e^{i\alpha})e^{-in\alpha}dzd\alpha,\
n\in\Z.$$From the Plancherel theorem for $G$ (see page 183,
\cite{S}) we have for $f\in L^2(G)$,
$$\|f\|^2_2=\int_0^{\infty}\|\pi_{\lambda}(f)\|^2_{HS}\lambda
d\lambda,$$ where $\|.\|_{HS}$ denotes the Hilbert-Schmidt norm.\\

For $g_1,g_2\in G$, the two sided translate, $^{g_1}f^{g_2}$ of
$f$ is the function defined by $^{g_1}f^{g_2}(g)=f(g_1^{-1}gg_2).$
For $f\in L^1(G)\cap L^p(G)$, let $S=\{a>0:\pi_a(f)=0\},\ X=Span\
\{^{g_1}f^{g_2}:g_1,g_2\in G\},\ S'=\{\lambda>0:$ Range of
$\pi_{\lambda}(f)$ is not dense$\}$ and $V_f$ be the closed
subspace spanned by the right translates of $f$ in $L^p(G)$.
\begin{thm}
Let $f\in L^1(G)\cap L^p(G)$. \begin{enumerate} \item For
$\frac{4}{3-\alpha}\leq p<2$, if $S=\{a>0: \pi_a(f)=0\}$ is such
that $\CP^{\alpha}(S)<\infty$ for $0\leq\alpha<1$, then $X=span\{
^{g_1}f^{g_2}: g_1, g_2\in M(2)\}$ is dense in $L^p(M(2))$. \item
If $f$ is radial in the $\mathbb{R}^2$ variable and
$\CP^{\alpha}(S')<\infty$ for some $0\leq\alpha<1$, then
$V_f=L^p(M(2))$ provided $\frac{4}{3-\alpha}\leq p\leq 2$.
\end{enumerate}
\end{thm}
\begin{proof}
To prove part (1), we proceed as in the proof of Theorem 2.1 in \cite{NR}.
It is enough to prove $L^p(G/K)\subseteq\ol{X}$.\\

For given $a,\epsilon>0$, there exists constants
$c_1,c_2,...,c_m$, $w\in H$ and elements $x_1,x_2,...x_m\in\ G$
such that $\|\sum_{j=1}^{m}c_j\pi_a(x_j)v_0-w\|<\epsilon$, where
$v_0$ is $K$-fixed vector $v_o\equiv1\in H$. Define
$F_a=\sum_{j=1}^mc_jf^{x_j^{-1}}$. Then $\pi_a(F_a)v_0\neq0$. Let
$$F_a^{\#}(x)=\int_KF_a(xk)dk,\ x\in G.$$ Then whenever $\pi_a(f)\neq0$,
as in the proof of Theorem 2.1 in \cite{NR} we have a right
$K$-invariant function $F_a^{\#}$ which can be considered as a
function on $\R^2$, that is $F_a^{\#}\in L^1(\R^2)\cap L^p(\R^2)$
such that its Euclidean Fourier transform is not identically zero on the sphere $C_a=\{x\in\R^2:\|x\|=a\}.$\\

Define $S_1=\cap_{a\in S^c}\{r>0:\wh{F}_a^{\#}\equiv0$ on $C_r\}$.
Then $S_1\subset S$. We have $\overline{Span\{^gF_a^{\#}:g\in
G,a\in S_1^c\}}\subseteq\overline{Span\{^{g_1}f^{g_2}:g_1,g_2\in
G\}}$. Also using Theorem \ref{thmMYRr}, $\overline{Span\{^gF_a^{\#}:g\in
G,a\in S_1^c\}}=L^p(G/K)$. Thus
$L^p(G/K)\subseteq\overline{Span\{^{g_1}f^{g_2}:g_1,g_2\in
G\}}=\ol{X}$.\\

To prove part(2), we proceed as in the proof of (c) of Theorem 3.2
in \cite{NR}. Let $\phi(z)e^{im_0\alpha}\in L^q\cap
L^{\infty}(M(2))$ kill all the functions in $V_f$ where
$\frac{1}{p}+\frac{1}{q}=1$. Then $f$ being radial in the
$\R^2$-variable we are led to the convolution equation
$f_m*_{\R^2}\phi_m=0$ where $\phi_m$ is defined by
\begin{equation*}
    \phi_m(z)=\int_0^{2\pi}\phi(e^{i\alpha}z)e^{i(m_0+m)\alpha}d\alpha.
\end{equation*}
and $f_m$ is defined by
\begin{equation*}
    f_m(z)=\int_{S^1}f(z,e^{i\alpha})e^{-im\alpha}d\alpha
\end{equation*}

Taking Fourier transform we obtain that $supp\ \wh{\phi}_m$ is
contained in $\{z\in\mathbb{R}^2:\|z\|\in S\}$. Proceeding as in the proof of Theorem \ref{thmMYRr},
we have $<\phi_m,\psi>=0$ for all $\psi\in C_c^{\infty}(\R^2)$ and $m$. Thus $\phi_m\equiv0$ for all $m$.\\
\end{proof}

%%%%%%%%%%%%%%%%%%%%%%%%%%%%%% Acknowledgements %%%%%%%%%%%%%%%%%%%%%%%%%%%%%%%%%%%%%

\section{Acknowledgements}
The author wishes to thank both the referees for several
suggestions and remarks which improved the presentation of the
paper. The author wishes to thank Prof. Malabika Pramanik for
useful discussions. The author is also grateful to Dr. E. K.
Narayanan for his constant encouragement and for the many useful
discussions during the course of this work. The author also wishes
to thank UGC-CSIR for financial support. This work is supported in
part by UGC Centre for Advanced Studies.

\end{document}